\documentclass{amsart}
\usepackage{graphicx,color}
\usepackage{latexsym}
\usepackage{amsfonts}
\usepackage[all]{xy}
\usepackage{amssymb, mathrsfs, amsfonts, amsmath}
\usepackage{amsbsy}
\usepackage{amsfonts}
\newtheorem*{acknowledgement}{Acknowledgement}
\newtheorem{corollary}{Corollary}

\newtheorem{definition}{Definition}
\newtheorem{lemma}{Lemma}

\newtheorem{remark}{Remark}
\newtheorem{theorem}{Theorem}
\newtheorem{example}{Example}
\numberwithin{equation}{section}

\title[Conformal vector fields on Lie groups]{Conformal vector fields on Lie groups}
\author{Adriana Araujo Cintra, Zhiqi Chen and Benedito Leandro Neto}

\address{Universidade de Federal de Goi\'as, Regional Jata\'i, Endere\c co, BR 364, km 195, 3800, CEP 75801-615, Jata\'i, Goi\'as, Brazil.}

\email{adriana.cintra@ufg.br}

\address{School of Mathematical Sciences and LPMC, Nankai University, 300071 Tianjin, People's Republic of China}

\email{chenzhiqi@nankai.edu.cn}

\address{Universidade de Federal de Goi\'as, Regional Jata\'i, Endere\c co, BR 364, km 195, 3800, CEP 75801-615, Jata\'i, Goi\'as, Brazil.}

\email{bleandroneto@gmail.com}

\allowdisplaybreaks
\numberwithin{equation}{section}
\numberwithin{theorem}{section}

\date{\today}

\begin{document}

\newcommand{\spacing}[1]{\renewcommand{\baselinestretch}{#1}\large\normalsize}
\spacing{1.2}

\begin{abstract}
In this paper, we investigated the behavior of left-invariant conformal vector fields on Lie groups with left-invariant pseudo-Riemannian metrics. First of all, we prove that conformal vector fields on pseudo-Riemannian unimodular Lie groups are Killing. Then we obtain a necessary condition for a pseudo-Rimennian non-unimodular Lie group admitting a non-Killing conformal vector field. Finally, we give examples of non-Killing conformal vector fields and Yamabe solitons on non-unimodular Lorentzian Lie groups based on the above study.
\end{abstract}

\maketitle

\section{Introduction}
A transformation of a Riemannian manifold is say to be conformal if it preserves the angle defined by the Riemannian metric. Furthermore, we said that a Riemannian manifold admits a conformal vector field $X$ if $$\mathfrak{L}_{X}g=2\rho g,$$ where $g$ is the Riemannian metric, $\mathfrak{L}_{X}g$ is the Lie derivative and $\rho$ is a smooth function. The existence of such a function might give some information about the topological structure of the Riemannian manifold \cite{Caminha,Obata}. An important class of conformal vector fields are {\it Killing vector fields}, i.e., vector fields such that $\mathfrak{L}_{X}g=0$. Killing vector fields are generators of isometries, which provides a close link between the geometry of a manifold $M$ and the algebra of $I(M)$, the set of all isometries in $M$ (see \cite{Oneil}). Yamabe soliton vector fields are another important class of conformal vector fields.

A {\it Yamabe soliton} is a complete manifold $M^{n}$ with a metric $\langle\cdot,\cdot\rangle$, a vector field $X\in\mathfrak{X}(M)$ and a constant $\lambda$ that satisfies the equation:
\begin{eqnarray}\label{yamabesoliton}
(R-\lambda)\langle\cdot,\cdot\rangle=\frac{1}{2}\mathfrak{L}_{X}\langle\cdot,\cdot\rangle,
\end{eqnarray}
where
\begin{eqnarray}\label{liederivative}
(\mathfrak{L}_{X}g)(Y,Z)=X\langle Y, Z\rangle-\langle [X,Y], Z\rangle - \langle Y, [X, Z]\rangle,
\end{eqnarray}
for any $X,Y,Z\in\mathfrak{X}(M)$ and $R$ is the scalar curvature of the metric $\langle\cdot,\cdot\rangle$. If $\lambda>0$, $\lambda=0$ or $\lambda<0$, then we have, respectively, a Yamabe soliton {\it shrinking}, {\it steady} and {\it expanding}. A Yamabe soliton is called {\it non-trivial} if it admits a soliton vector field which is not Killing. In particular, the scalar curvature is constant if the Yamabe soliton is trivial. If $X=\nabla f$ for some smooth potential function $f$ on $M$, we call it a {\it gradient Yamabe soliton} (see \cite{Leandro} for instance).

The Yamabe solitons are a special type of solution for the Yamabe flow
\begin{eqnarray*}
\frac{\partial}{\partial t}g(t)=-Rg(t).
\end{eqnarray*}
The Yamabe flow is introduced as an attempt to solve the Yamabe problem \cite{Hamilton}. The solitons have become an important tool to explore the geometric flows. For these reasons, Ricci solitons and Yamabe solitons have been extensively explored in the past years \cite{BarbosaRibeiro,Calvino,CC,CSZ,ChowKnopf,Nazareno,Jablonski,Shu}. It is worth to point out that $2$-dimensional Yamabe soliton and Ricci soliton are the same. The soliton solutions do not necessarily exist in the Lorentzian case (see \cite{Calvino}). In addition, generalizations of solitons were considered and investigated \cite{ChenLiangYi,Leandro}. Chen, Liang and Yi \cite{ChenLiangYi} considered a generalization of Ricci solitons on Lie groups. In fact, they proved that exist many non-trivial $m$-quasi Einstein metrics in a solvable quadratic Lie group. In \cite{Leandro}, Leandro Neto and Pina de Oliveira considered a generalization of gradient Yamabe solitons. They proved that connected generalized quasi Yamabe gradient solitons are, in fact, $m$-quasi Yamabe gradient solitons.

A Lie group is a set which has a structure of manifold and group at the same time, where the group operations are smooth. Of central importance is the relationship between the Lie group and its Lie algebra (see \cite{GallierQuaintance,Milnor,Oneil,Warner}). Lie groups are very useful to give examples. As Milnor \cite{Milnor} said, ``when studying relationships between curvature of a complete Riemannian manifold and other topological or geometric properties, it is useful to have many examples". On the other hand, any left-invariant conformal vector field on a Lie group with a left-invariant Riemannian metric is Killing.

This paper is to study left-invariant non-Killing conformal vector fields on Lie groups with left-invariant pseudo-Riemannian metrics. It is point out in \cite{Calvino} that there are 3-dimensional Lie groups with left-invariant Lorentzian metrics admitting conformal vector fields which are not left-invariant. Without special notes, a vector field in this paper means a left-invariant vector field. Firstly, we have the following theorem.

\begin{theorem}\label{theorem2}
Let $G$ be an $n$-dimensional unimodular Lie group with a left-invariant pseudo-Riemannian
metric $\langle\cdot,\cdot\rangle$ and let $\mathfrak{g}$ be the unimodular Lie algebra of left-invariant vector fields. If $X\in \mathfrak{g}$ is a conformal vector field, then $X$ is a Killing vector field.
\end{theorem}

Since every nilpotent and compact groups are unimodular, by Theorem \ref{theorem2}, every left-invariant conformal vector field on such Lie groups is Killing. Moreover, by Theorem \ref{theorem2},
\begin{corollary}\label{corollary1}
Let $G$ be a unimodular Lie group with a left-invariant pseudo-Riemannian metric $\langle\cdot,\cdot\rangle$ and let $\mathfrak{g}$ be the unimodular Lie algebra of left-invariant vector fields. If $X\in \mathfrak{g}$ is a vector field satisfying (\ref{yamabesoliton}), then $X$ is a Killing vector field. That is, left-invariant Yamabe solitons on unimodular Lie groups are trivial.
\end{corollary}

From \cite{Calvino}, we know that there are non-Killing conformal vector fields on $3$-dimensional non-unimodular Lie groups with lorentzian metrics. Here we have the following theorem.

\begin{theorem}\label{prop1}
Let $G$ be an $n$-dimensional non-unimodular Lie group with a left-invariant pseudo-Riemannian
metric $\langle\cdot,\cdot\rangle$ of signature $(p,q)$ and let $\mathfrak{g}$ be the non-unimodular Lie algebra of left-invariant vector fields. If $x\in \mathfrak g$ is a non-Killing conformal vector field, then $\dim C(\mathfrak g)\leq \min(p,q)$ and $\dim [\mathfrak g,\mathfrak g]\geq \dim \mathfrak g-\min(p,q)$.
\end{theorem}

Finally, we classify non-unimodular Lie groups with Lorentzian metrics admitting non-Killing conformal vector fields in dimensions $2$ and $3$, and then give examples in higher dimensions.

\section{Preliminaries}

\begin{definition}
A {\it Lie group} $G$ is a differentiable manifold which is also endowed with a group structure such that the map $G\times G\rightarrow G$ defined by $(\sigma, \tau) \mapsto \sigma\tau^{-1}$ is smooth.
\end{definition}

\begin{definition}
A {\it Lie algebra} $\mathfrak{g}$ is a vector space $L$ together with a bilinear operator $[\cdot,\cdot]:\mathfrak{g}\times \mathfrak{g}\rightarrow \mathfrak{g}$ (called the bracket) such that for all $X, Y, Z\in  \mathfrak{g}$,
\begin{enumerate}
\item[(1)] [X,Y]=-[Y,X]\quad\mbox{(skew-symmetric)},
\item[(2)] [[X,Y],Z]+[[Y,Z],X]+[[Z,X],Y]=0\quad\mbox{(Jacobi identity)}.
\end{enumerate}
\end{definition}

\begin{definition}
For any $\sigma\in G$, the left translation $l_{\sigma}$ is the diffeomorphism of $G$ defined by
\begin{eqnarray*}
l_{\sigma}(\tau)=\sigma\tau,
\end{eqnarray*}
for any $\tau\in G$. A vector field $X$ on $G$ is called {\it left-invariant} if $X$ is $l_{\sigma}$-related to itself for each $\sigma\in G$; that is,
\begin{eqnarray*}
dl_{\sigma}\circ X=X\circ l_{\sigma}.
\end{eqnarray*}
\end{definition}

\begin{definition}
A metric $\langle\cdot,\cdot\rangle$ on a Lie group $G$ is called left-invariant if
\begin{eqnarray*}
\langle X,Y\rangle_{\tau}=\langle (dl_{\sigma})_{\tau}(X),(dl_{\sigma})_{\tau}(Y)\rangle_{\sigma\tau},
\end{eqnarray*}
for all $\sigma,\tau\in G$ and $X, Y\in \mathfrak{g}$.
\end{definition}

Let $G$ be a Lie group with the Lie algebra $\mathfrak{g}$ of left-invariant vectors fields on $G$. For any $X\in \mathfrak{g}$, {\it the adjoint map} $ad_{X}: \mathfrak{g}\rightarrow \mathfrak{g}$ sends $Y\longmapsto [X,Y]$, i.e., $ad_{X}\{Y\}=[X,Y]$.

\begin{definition}\label{center}
The {\it center} $C(\mathfrak{g})$ of a Lie algebra $\mathfrak{g}$ is the set of all $X\in\mathfrak{g}$ such that $ad_{X}=0$.
\end{definition}

\begin{definition}\label{unimodular}
A Lie group $G$ is {\it unimodular} if $tr\{ad_{X}\}=0$ for any $X\in\mathfrak{g}$.
\end{definition}

\begin{definition}\label{ideal}
If $\mathfrak{g}$ is a Lie algebra, a subalgebra $\mathfrak{h}$ of $\mathfrak{g}$ is a (linear) subspace of $\mathfrak{g}$ such that $[u, v]\in\mathfrak{h}$, for all $u, v \in\mathfrak{h}$. If $\mathfrak{h}$ is a (linear) subspace of $\mathfrak{g}$ such that $[u, v]\in\mathfrak{h}$ for all
$u\in \mathfrak{h}$ and all $v\in \mathfrak{g}$, we say that $\mathfrak{h}$ is an ideal in $\mathfrak{g}$.
\end{definition}

\begin{definition}\label{commutatorideal}
Given two subsets $\mathfrak{a}$ and $\mathfrak{b}$ of a Lie algebra $\mathfrak{g}$, denote by $[\mathfrak{a},\mathfrak{b}]$ the subspace of $\mathfrak{g}$ consisting of all linear combinations $[a,b]$ for any $a\in\mathfrak{a}$ and $b\in\mathfrak{b}$. In particular, $[\mathfrak{g}, \mathfrak{g}]$ is an ideal in $\mathfrak{g}$ called the {\it commutator ideal} of $\mathfrak{g}$.
\end{definition}

Let $G$ be a Lie group with the Lie algebra $\mathfrak g$ and let $\langle\cdot,\cdot\rangle$ be a pseudo-Riemannian metric on $G$. Assume that $\nabla$ is the Levi-Civita connection associated with $\langle\cdot,\cdot\rangle$. Then,
\begin{eqnarray}\label{liebrackets}
[X,Y]=\nabla_{X}Y-\nabla_{Y}X.
\end{eqnarray}
Moreover if $\langle\cdot,\cdot\rangle$ is left-invariant on $G$ (see \cite{GallierQuaintance,Oneil}), then for left-invariant vector fields $X, Y$, we have
\begin{eqnarray*}
\langle X,Y\rangle_{\sigma}=\langle X(\sigma),Y(\sigma)\rangle_{\sigma}=\langle (dl_{\sigma})_{e}X(e),(dl_{\sigma})_{e}Y(e)\rangle_{\sigma}=\langle X(e),Y(e)\rangle_{e}=\langle X,Y\rangle_{e},
\end{eqnarray*}
where $e$ is the identity element of $G$. It shows that the map $\sigma\mapsto\langle X,Y\rangle_{\sigma}$ is constant. Thus, for any vector field $Z$ on $\mathfrak{g}$,
\begin{eqnarray*}\label{leftinvariantmetric0}
Z\langle X,Y\rangle=0.
\end{eqnarray*}
This means that, for a left-invariant metric $\langle\cdot,\cdot\rangle$ on $G$, we have
\begin{eqnarray}\label{leftinvariantmetric}
\langle \nabla_{Z}X, Y\rangle +\langle X, \nabla_{Z}Y\rangle=0,
\end{eqnarray}
for any $X,Y,Z\in \mathfrak{g}$.
By (\ref{liebrackets}) and (\ref{leftinvariantmetric}),
\begin{eqnarray*}\label{connection}
\langle\nabla_{X}Y,Z\rangle=\frac{1}{2}(\langle[X, Y], Z\rangle - \langle[Y, Z], X\rangle + \langle[Z, X], Y\rangle),
\end{eqnarray*}
where $X, Y, Z$ are all left-invariant vector fields. Assume that $X\in\mathfrak{g}$ is a conformal vector field, i.e.
\begin{eqnarray}\label{conformalvectorfield}
\mathfrak{L}_{X}\langle\cdot,\cdot\rangle=2\rho\langle\cdot,\cdot\rangle.
\end{eqnarray}
It follows that,
\begin{eqnarray*}
0=\mathfrak{L}_{X}\langle X,X\rangle=2\rho|X|^{2}.
\end{eqnarray*}
If $\langle\cdot,\cdot\rangle$ is a left-invariant Riemannian metric, then $\rho=0$ or $X\equiv0$. That is, $X$ is Killing or $X$ is trivial. For this reason, we focus on a left-invariant pseudo-Riemannian metric $\langle\cdot,\cdot\rangle$.
\begin{lemma}
Let notations be as above. If $X$ is a non-Killing conformal vector field, then $X$ is a lightlike vector field, i.e., $\langle X,X\rangle=0$.
\end{lemma}
\begin{proof}
Otherwise, assume that $X$ is spacelike $\langle X,X\rangle > 0$ or timelike $\langle X,X\rangle < 0$. By (\ref{conformalvectorfield})
\begin{eqnarray*}
0=\mathfrak{L}_{X}\langle X,X\rangle=2\rho\langle X,X\rangle,
\end{eqnarray*}
which implies that $\rho=0$. That is, $X$ is trivial, i.e., a Killing vector field. Hence, we have that $X$ is a lightlike vector field.
\end{proof}

\section{Proof of Theorems \ref{theorem2} and \ref{prop1}}\label{sectionunimodular}
Let $G$ be a Lie group with the Lie algebra $\mathfrak g$ and let $\langle\cdot,\cdot\rangle$ be a left-invariant pseudo-Riemannian metric of signature $(p,q)$ on $G$. First of all, we have
\begin{eqnarray}\label{eq3}
&& \langle\{ad_{X}+ ad_{X}^{*}\}\{Y\},Z\rangle \nonumber\\
&=&\langle ad_{X}\{Y\},Z\rangle+\langle ad_{X}^{*}\{Y\}, Z\rangle = \langle ad_{X}\{Y\},Z\rangle+\langle Y, ad_{X}\{Z\}\rangle\nonumber\\
&=& \langle [X,Y], Z\rangle + \langle Y, [X,Z]\rangle,
\end{eqnarray}
where $\langle ad_{X}\{Y\}, Z\rangle = \langle Y,ad^{*}_{X}\{Z\}\rangle$ for all $X,Y,Z\in\mathfrak{g}$.
Hence, from (\ref{liederivative}) and (\ref{eq3})  we have
\begin{eqnarray*}\label{adjointrepre}
ad_{X}+ ad_{X}^{*}= -\mathfrak{L}_{X} \langle\cdot,\cdot\rangle.
\end{eqnarray*}
If $X\in \mathfrak g$ is a conformal vector field, by (\ref{conformalvectorfield}), then
\begin{eqnarray}\label{algebricconformalvectorfield}
ad_{X}+ad^{*}_{X}=-2\rho \langle\cdot,\cdot\rangle.
\end{eqnarray}

{\it Proof of Theorem \ref{theorem2}.}
Let $\{e_1,\cdots,e_s,e_{s+1},\cdots,e_{n}\}$ be an orthogonal basis of $\mathfrak{g}$ corresponding to the pseudo-Riemannian metric $\langle\cdot,\cdot\rangle$ of signature $(p,q)$. That is,
\begin{eqnarray*}
\langle e_{i},e_{j}\rangle=\left\{
\begin{array}{ll}
0,& \quad\mbox{if}\quad i\neq j,\\
1,& \quad\mbox{if}\quad 1\leq i,j \leq p,\\
-1,& \quad\mbox{if}\quad p+1\leq i,j\leq n=p+q.\\
\end{array}
\right.
\end{eqnarray*}
From (\ref{algebricconformalvectorfield}) we have that
\begin{eqnarray*}
-2\rho\langle e_{i},e_{i}\rangle = \langle\{ad_{X} + ad^{*}_{X}\}\{e_{i}\},e_{i}\rangle = 2\langle ad_{X}\{e_{i}\},e_{i}\rangle
\end{eqnarray*}
and
\begin{eqnarray*}
-2\rho\langle e_{i},-e_{i}\rangle = \langle\{ad_{X} + ad_{X}^{*}\}\{e_{i}\},-e_{i}\rangle = 2\langle ad_{X}\{e_{i}\},-e_{i}\rangle.
\end{eqnarray*}
Therefore,
\begin{eqnarray*}
-n\rho&=&-\rho\left(\sum_{i=1}^p \langle e_{i},e_{i}\rangle+\sum_{i=p+1}^n \langle e_{i},-e_{i}\rangle\right)\nonumber\\
&=& \sum_{i=1}^p \langle ad_{X}\{e_{i}\},e_{i}\rangle+\sum_{i=p+1}^n \langle ad_{X}\{e_{i}\},-e_{i}\rangle \\
&=& tr\{ad_{X}\}=0.
\end{eqnarray*}
Then $\rho=0$. That is, X is a Killing vector field.
$\hfill\Box$

{\it Proof of Theorem \ref{prop1}.}
For the first part, assume that $\dim C(\mathfrak g)=s > \min(p,q)$, and assume that $\{e_1,\cdots,e_s\}$ is a basis of $C(\mathfrak g)$. For any $1\leq i,j\leq \min(p,q)$,
$$\mathfrak{L}_{X}\langle e_i,e_j\rangle=-\langle [X,e_i], e_j\rangle-\langle e_i, [X, e_j]\rangle=0=2\rho \langle e_i,e_j\rangle.$$
Since $\rho\not=0$, we have $$\langle e_i,e_j\rangle =0,\quad \forall 1\leq i,j\leq s.$$
That is, we have an $s$-dimensional isotropy subspace of $\langle\cdot,\cdot\rangle$. But for a pseudo-Riemannian metric of signature $(p,q)$, we only can have a $\min(p,q)$-dimensional isotropy subspace. So we must have $\dim C(\mathfrak g)\leq \min(p,q)$.

For the second part, assume that $\dim [\mathfrak g,\mathfrak g]< \dim \mathfrak g-\min(p,q)$. It follows that $\dim[\mathfrak g,\mathfrak g]^\bot=t>\min(p,q)$, where $[\mathfrak g,\mathfrak g]^\bot$ is the orthogonal complement of $[\mathfrak g,\mathfrak g]$ corresponding to $\langle\cdot,\cdot\rangle$. Assume that $\{e_1,\cdots,e_t\}$ is a basis of $[\mathfrak g,\mathfrak g]^\bot$. For any $1\leq i,j\leq t$,
$$\mathfrak{L}_{X}\langle e_i,e_j\rangle=-\langle [X,e_i], e_j\rangle-\langle e_i, [X, e_j]\rangle=0=2\rho \langle e_i,e_j\rangle.$$
Since $\rho\not=0$, we have $$\langle e_i,e_j\rangle=0,\quad \forall 1\leq i,j\leq t.$$
That is, we have a $t$-dimensional isotropy subspace of $\langle\cdot,\cdot\rangle$. Similar to the first part, we have $\dim [\mathfrak g,\mathfrak g]\geq \dim \mathfrak g-\min(p,q)$.
$\hfill\Box$

\section{Non-Killing conformal vector fields on non-unimodular Lie groups with left-invariant Lorentzian metrics}\label{nonunimodular}

Let $G$ be a non-unimodular Lie group with the Lie algebra $\mathfrak{g}$ and let $\langle\cdot,\cdot\rangle$ be a Lorentzian metric on $G$. Assume that $\mathfrak g$ admits a non-Killing conformal vector field $X$. By Theorem~\ref{prop1}, $$\dim [\mathfrak g,\mathfrak g]\geq \dim \mathfrak g-1.$$
Furthermore, if $\mathfrak g$ is solvable, we know that $\dim [\mathfrak g,\mathfrak g]<\dim \mathfrak g$. It follows that $$\dim [\mathfrak g,\mathfrak g]=\dim \mathfrak g-1.$$
\begin{lemma}\label{lemma2}
Let notations be as above. Then the restriction of $\langle\cdot,\cdot\rangle$ on $[\mathfrak g,\mathfrak g]$ is degenerate.
\end{lemma}
\begin{proof}
Assume that the restriction of $\langle\cdot,\cdot\rangle$ on $[\mathfrak g,\mathfrak g]$ is non-degenerate. Then $\dim [\mathfrak g,\mathfrak g]^\bot=1$ and the restriction of $\langle\cdot,\cdot\rangle$ on $[\mathfrak g,\mathfrak g]^\bot$ is non-degenerate. Let $\{e\}$ be a basis of $[\mathfrak g,\mathfrak g]^\bot$. Thus $\langle e,e\rangle \not=0$. By $$\mathfrak{L}_{X}\langle e,e\rangle=-\langle [X,e], e\rangle-\langle e, [X, e]\rangle=0=2\rho \langle e,e\rangle$$ and $\rho\not=0$, we have $\langle e,e\rangle=0$. It is a contradiction. Thus the restriction of $\langle\cdot,\cdot\rangle$ on $[\mathfrak g,\mathfrak g]$ is degenerate.
\end{proof}

By Lemma~\ref{lemma2}, we will classify solvable Lie groups with Lorentzian metrics admitting non-Killing conformal vector fields in dimensions 2 and 3.

\begin{theorem}
Let $G$ be a 2-dimensional Lie group with the Lie algebra $\mathfrak g$ and let $\langle\cdot,\cdot\rangle$ be a Lorentzian metric on $G$. Assume that $\mathfrak g$ admits a non-Killing conformal vector field $X$. Then there is a basis $\{e_{1},e_{2}\}$ of $\mathfrak g$ such that the non-zero bracket is given by \begin{eqnarray*}
[e_{1},e_{2}]=e_{2}
\end{eqnarray*}
and the Lorentzian metric associated with the basis is defined by
\begin{eqnarray*}
\langle\cdot,\cdot\rangle=\left(
\begin{array}{cc}
0 & 1\\
1 & 0\\
\end{array}
\right).
\end{eqnarray*}
In particular, $X=x_{1}e_{1}$.
\end{theorem}
\begin{proof}
First of all, $\mathfrak g$ is solvable if $\dim \mathfrak g=2$. Since $\mathfrak g$ admits a non-Killing conformal vector field $X$, we have $\dim [\mathfrak g,\mathfrak g]=1$. Thus there is a basis $\{e_{1},e_{2}\}$ of $\mathfrak g$ such that
\begin{eqnarray*}
[e_{1},e_{2}]=e_{2},
\end{eqnarray*}
where $\{e_2\}$ is a basis of $[\mathfrak g,\mathfrak g]$. By Lemma~\ref{lemma2}, we can assume that the Lorentzian metric $\langle\cdot,\cdot\rangle$ on $\mathfrak{g}$ associated with the basis is defined by
\begin{eqnarray*}\label{matrix000}
\langle\cdot,\cdot\rangle=\left(
\begin{array}{cc}
a & 1\\
1 & 0\\
\end{array}
\right).
\end{eqnarray*}
Replacing $e_1$ by $e_1-\frac{a}{2}e_2$ if necessary, we can assume that $a=0$. Assume that $X=x_{1}e_{1}+x_{2}e_{2}$ is a left-invariant conformal vector field on $\mathfrak{g}$. By
\begin{eqnarray}\label{matrix10000}
0=\mathfrak{L}_{X}\langle\cdot,\cdot\rangle-2\rho\langle\cdot,\cdot\rangle=\left(
\begin{array}{cc}
2x_{2} & -x_{1}-2\rho\\
-x_{1}-2\rho & 0\\
\end{array}
\right),
\end{eqnarray}
we have $x_{2}=0$. Therefore, $X=x_{1}e_{1}$ is a non-Killing left-invariant lightlike conformal vector field with the conformal function given by $\rho=\frac{-x_{1}}{2}$.
\end{proof}
\begin{remark}
We know that $2$-dimensional Yamabe solitons and Ricci solitons are the same. A straightforward computation shows that the curvature $R=0$. Then, from (\ref{yamabesoliton}) the soliton will be non-trivial if and only if it is a shrinking or expanding soliton.
\end{remark}

\begin{theorem}
Let $G$ be a 3-dimensional Lie group with the Lie algebra $\mathfrak g$ and let $\langle\cdot,\cdot\rangle$ be a Lorentzian metric on $G$. Assume that $\mathfrak g$ admits a non-Killing conformal vector field $X$. Then there is a basis $\{e_{1}$, $e_{2}$, $e_{3}\}$ of $\mathfrak g$ such that the non-zero brackets are given by
\begin{eqnarray}
[e_{1}, e_{3}]=\alpha e_{1}+\beta e_{2};\quad [e_{2}, e_{3}]=2\alpha e_{2} \quad\mbox{and}\quad
[e_{1}, e_{2}]=0,
\end{eqnarray}
where $\alpha\not=0$, and the Lorentzian metric associated with the basis is defined by
\begin{eqnarray*}
\langle\cdot,\cdot\rangle=\left(
\begin{array}{ccc}
1 & 0 & 0 \\
0 & 0 & -1\\
0 & -1 & 0\\
\end{array}
\right).
\end{eqnarray*}
In particular,
\begin{eqnarray}\label{campoluz}
X=\left\{
\begin{array}{lcc}
x_{3}e_{3}\quad(\mbox{with}\quad\rho=x_{3}\alpha),\quad\mbox{for}\quad\beta=0;\\
x_{1}(e_{1}-\frac{\beta}{\delta}e_{2}-\frac{\delta}{2\beta}e_{3})\quad(\mbox{with}\quad\rho=-x_{1}\frac{\delta^{2}}{2\beta}),\quad\mbox{for}\quad\beta\neq0.
\end{array}
\right.
\end{eqnarray}
\end{theorem}
\begin{proof}
First we know that $\mathfrak g$ is non-unimodular since $\mathfrak g$ admits a non-Killing conformal vector field $X$. Then $\mathfrak g$ is solvable by $\dim \mathfrak g=3$. Thus we know dim$[\mathfrak g,\mathfrak g]=2$. Since $\mathfrak g$ is solvable, we know $[\mathfrak g,\mathfrak g]$ is nilpotent. Thus $[\mathfrak g,\mathfrak g]$ is abelian. By Lemma~\ref{lemma2}, we know that there exist a basis $\{e_{1}$, $e_{2}$, $e_{3}\}$ such that the Lorentzian metric associated with the basis is given by \begin{eqnarray*}\label{matrix}
\langle\cdot,\cdot\rangle=\left(
\begin{array}{ccc}
1 & 0 & 0 \\
0 & 0 & -1\\
0 & -1 & 0\\
\end{array}
\right).
\end{eqnarray*}
Here the non-zero brackets must be
\begin{eqnarray*}
[e_{1}, e_{3}]=\alpha e_{1}+\beta e_{2}, \quad [e_{2}, e_{3}]=\gamma e_{1}+\delta e_{2},\quad [e_{1}, e_{2}]=0.
\end{eqnarray*}
Since $\mathfrak g$ is non-unimodular, we have $\alpha+\delta\neq0$ (see \cite{Calvino, Cordero, Milnor}). Assume that $X=\displaystyle\sum_{i=1}^{3}x_{i}e_{i}$ is a left-invariant conformal vector field. Therefore,
\begin{eqnarray*}\label{matrix1}
0=\mathfrak{L}_{X}\langle\cdot,\cdot\rangle-2\rho\langle\cdot,\cdot\rangle=\left(
\begin{array}{ccc}
2(x_{3}\alpha-\rho) & x_3\gamma & -x_{1}\alpha-x_{3}\beta \\
x_3\gamma & 0 & -x_{3}\delta+2\rho \\
-x_{1}\alpha - x_{3}\beta & -x_{3}\delta+2\rho & 2(x_{1}\beta+x_{2}\delta)\\
\end{array}
\right).
\end{eqnarray*}
If $x_{3}=0$ or $\delta=0$, then $\rho=0$, i.e., $X$ is Killing. Since $X$ is non-Killing, we have $x_3\delta\not=0$. It implies that $\gamma=0$ and $\alpha=\frac{\delta}{2}\neq0$. Furthermore,
\begin{eqnarray*}
X=\left\{
\begin{array}{lcc}
x_{3}e_{3}\quad(\mbox{with}\quad\rho=x_{3}\alpha),\quad\mbox{for}\quad\beta=0;\\
x_{1}(e_{1}-\frac{\beta}{\delta}e_{2}-\frac{\delta}{2\beta}e_{3})\quad(\mbox{with}\quad\rho=-x_{1}\frac{\delta^{2}}{2\beta}),\quad\mbox{for}\quad\beta\neq0.
\end{array}
\right.
\end{eqnarray*}
It ends the theorem.
\end{proof}

\begin{remark}
Clearly $X$ defined by (\ref{campoluz}) is lightlike. By a straightforward computation we have that the scalar curvature $R=0$. This implies that the left-invariant Lorentzian Yamabe soliton with $X$ given by (\ref{campoluz}) is non-trivial if and only if it is shrinking or expanding. In fact, the classification of 3-dimensional Lorentzian Lie groups admitting non-Killing conformal vector fields is given in Theorem 7 of \cite{Calvino}. Here we give a simpler proof by our method.
\end{remark}

The following is to give some examples of non-Killing conformal vector fields and Yamabe solitons on $G$ of higher dimensions.

\begin{example}\label{example1}
Let $G$ be a connected $4$-dimensional non-unimodular with the Lie algebra $\mathfrak{g}$. Let $\{e_{1}$, $e_{2}$, $e_{3}, e_{4}\}$ be a basis of $\mathfrak g$ such that the brackets are defined by
\begin{eqnarray*}\label{frameorigem}
&&[e_{1}, e_{2}]=\alpha e_{3};\quad [e_{1}, e_{3}]=[e_{2}, e_{3}]=0;\quad[e_{1}, e_{4}]
=-\frac{1}{2}e_{1};\nonumber\\
&&[e_{2}, e_{4}]=-\frac{1}{2}e_{2}\quad\mbox{and}\quad[e_{3}, e_{4}]=-e_{3},
\end{eqnarray*}
where $\alpha\in\mathbb{R}$. Consider the Lorentzian metric on $\mathfrak{g}$ associated with the basis given by
\begin{eqnarray*}\label{matrix222}
\langle\cdot,\cdot\rangle=\left(
\begin{array}{cccc}
1 & 0 & 0 & 0 \\
0 & 1 & 0 & 0 \\
0 & 0 & 0 & -1 \\
0 & 0 & -1 & 0 \\
\end{array}
\right).
\end{eqnarray*}
Assume that $X=\displaystyle\sum_{i=1}^{4}x_{i}e_{i}$ is a left-invariant conformal vector field. Then,
\begin{eqnarray*}
0=\mathfrak{L}_{X}\langle\cdot,\cdot\rangle-2\rho\langle\cdot,\cdot\rangle=\left(
\begin{array}{cccc}
-x_{4}-2\rho & 0 & -2x_{3} & -\alpha x_{2}+\frac{x_{1}}{2} \\
0 & -x_{4}-2\rho & 0 & \alpha x_{1}+\frac{x_{2}}{2} \\
-2x_{3} & 0 & 0 & x_{4}+2\rho \\
-\alpha x_{2}+\frac{x_{1}}{2} & \alpha x_{1}+\frac{x_{2}}{2}  & x_{4}+2\rho & -2x_{3} \\
\end{array}
\right).
\end{eqnarray*}
Clearly, $x_{4}=0$ if, and only if, $\rho=0$, i.e., $X$ is a Killing vector field. 
In the following, assume that $x_{4}\neq0$. By $-\alpha x_{2}+\frac{x_{1}}{2}=0$ and $\alpha x_{1}+\frac{x_{2}}{2}=0$, we have $x_1=x_2=0$. Then,
\begin{eqnarray*}\label{campoluz2}
X=x_{4}e_{4},\quad\mbox{with}\quad\rho=\frac{-x_{4}}{2}.
\end{eqnarray*}
Clearly, the scalar curvature is given by
\begin{eqnarray*}
R=\frac{\alpha(1-\alpha)}{2}.
\end{eqnarray*}
Then the Damek-Ricci Lorentzian Yamabe soliton is trivial if, and only if, $\lambda=R=\frac{\alpha(1-\alpha)}{2}$. For this case, if $\alpha\not=0$ and $\alpha\not=1$, the Damek-Ricci Lorentzian Yamabe soliton can be shrinking, steady or expanding.
\end{example}

\begin{example}\label{example}
Let $\mathfrak g$ be a solvable Lie algebra of dimension $n$ such that $\dim [\mathfrak g,\mathfrak g]=n-1$ and $[\mathfrak g,\mathfrak g]$ is abelian. Assume that $\{e_1,\cdots,e_{n}\}$ is a basis of $[\mathfrak g,\mathfrak g]$, where $\{e_1,\cdots,e_{n-1}\}$ is a basis of $[\mathfrak g,\mathfrak g]$. The non-zero brackets are defined by $$[e_n,e_i]=\lambda_i e_i, \quad \text{ for any } 1\leq i\leq n-1,$$ where $\lambda_i\not=0$ for any $1\leq i\leq n-1$. Then $\mathfrak g$ is non-unimodular if and only if $\sum_{i=1}^{n-1}\lambda_i\not=0$. Consider the following Lorentzian metric on $\mathfrak{g}$ associated with the basis given by
\begin{eqnarray*}\label{matrixn}
\langle\cdot,\cdot\rangle=\left(
\begin{array}{ccccc}
1 & 0 & \cdots & 0 & 0 \\
0 & 1 & \cdots & 0 & 0 \\
\vdots & \vdots & \ddots & \vdots & \vdots \\
0 & 0 & \cdots & 0 & 1 \\
0 & 0 & \cdots & 1 & 0 \\
\end{array}
\right).
\end{eqnarray*}
Assume that $X=\displaystyle\sum_{i=1}^{n}x_{i}e_{i}$ is a left-invariant conformal vector field. For any $1\leq i\leq n-2$, $$0=\mathfrak{L}_{X}\langle e_i,e_n\rangle-2\rho\langle e_i,e_n\rangle=\lambda_ix_i.$$
That is, $x_i=0$ for any $1\leq i\leq n-2$. By $$0=\mathfrak{L}_{X}\langle e_{n},e_n\rangle-2\rho\langle e_{n},e_n\rangle=\lambda_{n-1}x_{n-1},$$ we have $x_{n-1}=0$. Thus $X=x_ne_n$. By $$0=\mathfrak{L}_{X}\langle e_{n-1},e_n\rangle-2\rho\langle e_{n-1},e_n\rangle=-\lambda_{n-1}x_{n}-2\rho,$$ we have $\rho=-\frac{\lambda_{n-1}x_n}{2}$.
For any $1\leq i\leq n-2$, by $$0=\mathfrak{L}_{X}\langle e_i,e_i\rangle-2\rho\langle e_i,e_i\rangle,$$ we know $\lambda_i=\frac{1}{2}\lambda_{n-1}$ if $x_n\not=0$.
Moreover, it is easy to check $X=x_ne_n$ is a non-Killing conformal vector field if $\lambda_i=\frac{1}{2}\lambda_{n-1}$ for any $1\leq i\leq n-2$.
\end{example}

\begin{example}
In fact, in Example~\ref{example}, we can delete that $[\mathfrak g,\mathfrak g]$ is abelian. That is,
\begin{enumerate}
   \item $\mathfrak g$ is a solvable Lie algebra of dimension $n$.
   \item $\dim [\mathfrak g,\mathfrak g]=n-1$.
   \item Let $\{e_1,\cdots,e_{n}\}$ be a basis of $[\mathfrak g,\mathfrak g]$, where $\{e_1,\cdots,e_{n-1}\}$ is a basis of $[\mathfrak g,\mathfrak g]$. The non-zero brackets are defined by $[e_n,e_i]=\lambda_i e_i$ where $\lambda_i\not=0$ for any $1\leq i\leq n-1$ and $\sum_{i=1}^{n-1}\lambda_i\not=0$.
\end{enumerate}
Consider the following Lorentzian metric on $\mathfrak{g}$ associated with the basis given by
\begin{eqnarray*}
\langle\cdot,\cdot\rangle=\left(
\begin{array}{ccccc}
1 & 0 & \cdots & 0 & 0 \\
0 & 1 & \cdots & 0 & 0 \\
\vdots & \vdots & \ddots & \vdots & \vdots \\
0 & 0 & \cdots & 0 & 1 \\
0 & 0 & \cdots & 1 & 0 \\
\end{array}
\right).
\end{eqnarray*}
Similar to the discussion in Example~\ref{example}, if $X$ is a left-invariant non-Killing conformal vector field, then $X=ae_n$ for some $a\not=0$. Moreover,
$\lambda_i=\frac{1}{2}\lambda_{n-1}$ for any $1\leq i\leq n-2$. It is easy to check that $X=ae_n$ is a non-Killing conformal vector field if $\lambda_i=\frac{1}{2}\lambda_{n-1}$ for any $1\leq i\leq n-2$. For this case, we have  $$[e_n,[e_i,e_j]]=[[e_n,e_i],e_j]+[e_i,[e_n,e_j]]=\lambda_i[e_i,e_j]+\lambda_j[e_i,e_j]=(\lambda_i+\lambda_j)[e_i,e_j]$$ for any $1\leq i,j\leq n-1$. That is,
$$ [e_i,e_j]=\beta_{ij}e_{n-1}, \text{ if } 1\leq i\not=j\leq n-2;\quad [e_i,e_{n-1}]=0, \text{ if } 1\leq i\leq n-2.
$$
This example can be considered as the higher dimensional case of the example~\ref{example1}.
\end{example}

\begin{acknowledgement}
The third author is grateful to Jos\'e Nazareno Gomes for bringing several references to his attention. 
\end{acknowledgement}


\begin{thebibliography}{BB}

\bibitem{BarbosaRibeiro} Barbosa, E., Ribeiro Jr, E.: {\em On conformal solutions of the Yamabe flow.} Arch. Math. 101, 79-89, (2013).


\bibitem{Calvino}Calvi\~no-Louzao, E., Seoane-Bascoy, J., V\'{a}zquez-Abal, M.E., V\'{a}zquez-Lorenzo, R.: {\em Three-dimensional homogeneous Lorentzian Yamabe solitons}. In: Abh. Math. Semin. Univ. Hambg. 82, 193-203, (2012).

\bibitem{Caminha}Caminha, A.:{\em The geometry of closed conformal vector fields on Riemannian spaces}. Bull. Braz. Math. Soc., New Series 42(2), 277-300, (2011).

\bibitem{CC} Cao, H.-D., Chen, Q.: {\em On locally conformally flat gradient steady Ricci solitons.} Trans. Am. Math. Soc. 2377–2391, (2012).

\bibitem{CSZ} Cao, H.-D., Sun, X., Zhang, Y.: {\em On the structure of gradient Yamabe solitons.} Math. Res. Lett. 19, 767-774, (2012).

\bibitem{Cordero} Cordero, L.A., Parker, Ph.: {\em Left-invariant Lorentzian metrics on 3-dimensional Lie groups.} Rend. Mat. VII 17, 129-155, (1997).

\bibitem{ChenLiangYi} Chen, Z., Liang, K., Yi, F: {\em Non-trivial m-quasi-Einstein metrics on quadratic Lie groups.} Arch. Math. 106, 391-399, (2016).


\bibitem{ChowKnopf} Chow, B., Knopf, D: {\em The Ricci flow: An introduction.} Mathematical surveys and monographs, ISSN 0076-5376; v.110, (2004).

\bibitem{Nazareno} Gomes, J.N., Wang, Q., and Xia, C.: {\em On the h-almost Ricci soliton.} arXiv:1411.6416v2 [math.DG] 3 May (2015).

\bibitem{GallierQuaintance} Gallier, J., Quaintance, J.: {\em Notes on Differential Geometry and Lie Groups.} Department of Computer and Information Science
University of Pennsylvania
Philadelphia, PA 19104, USA.

\bibitem{Hamilton} Hamilton, R.S: {\em The Ricci flow on surfaces. Mathematics and general relativity.} Contemp. Math., 71, (1988), 237-262.

\bibitem{Jablonski} Jablonski, M.: {\em Homogeneous Ricci solitons.} J. Reine Angew. Math. (Crelles Journal), v. 2015, n. 699, p. 159-182, (2015).

\bibitem{Leandro} Leandro Neto, B., Pina Oliveira, H.: {\em Generalized quasi Yamabe gradient solitons.} Differ. Geom. Appl. 49, 167-175, (2016).

\bibitem{Milnor} Milnor, J.: {\em Curvatures of Left Invariant Metrics on Lie Groups.} Adv. Math. 21, 293-329, (1976).

\bibitem{Obata} Obata, M.: {\em Certain conditions for a Riemannian manifold to be isometric with a sphere .} J. Math. Soc. Japan, Vol. 14, No. 3, (1962).

\bibitem{Oneil} O'Neill, B.: {\em Semi-Riemannian geometry with applications to relativity.} Academic Press (1983).

\bibitem{Shu} Shu, S.-Y.: {\em A note on compact gradient Yamabe solitons.} J. Math. Anal. Appl., Volume 388, Issue 2, Pages 725-726 (2012).

\bibitem{Warner} Warner, F. W.: {\em Foundations of differentiable manifolds and Lie groups.} Springer-Verlag (1971).


\end{thebibliography}
\end{document}